\documentclass{siamltex}
\usepackage{amsmath,verbatim}
\usepackage{amsfonts}
\usepackage{graphicx}
\usepackage{amssymb}
\usepackage{epstopdf}
\usepackage{graphicx}
\usepackage{cite}
\usepackage[usenames]{color}
\definecolor{ascolor}{rgb}{0,0.5,0}

\usepackage[normalem]{ulem}

\DeclareMathAccent{\maxvec}{\mathord}{letters}{"7E}
\DeclareMathOperator*{\argmin}{arg\,min}

\usepackage{subfigure}
\usepackage{cleveref}
\newtheorem{assumption}{Assumption}

\newtheorem{thm}{Theorem}[section]

\newtheorem{prop}[thm]{Proposition}

\newtheorem{rem}[thm]{Remark}
\numberwithin{equation}{section}
\title{An optimization-based Atomistic-to-Continuum coupling method\thanks{
		 Sandia National Laboratories is a multi-program laboratory
                managed and operated by Sandia Corporation, a wholly owned subsidiary of
                Lockheed Martin Corporation, for the U.S. Department of
                Energy's National Nuclear Security Administration under
                contract DE-AC04-94AL85000.}
}

\author{Derek Olson\thanks{University of Minnesota (\{olso4056,luskin,ashapeev\}@umn.edu).
DO was supported by the Department of Defense (DoD) through the National Defense Science \& Engineering
Graduate Fellowship (NDSEG) Program.  ML was supported in part by the NSF PIRE Grant OISE-0967140,
DOE Award DE-SC0002085, and AFOSR Award
FA9550-12-1-0187. AS was supported in part by the
DOE Award DE-SC0002085.}
\and
Pavel B. Bochev\thanks{Sandia National Laboratories,
Numerical Analysis and Applications, P.O. Box 5800, MS 1320,\break
Albuquerque, NM 87185-1320 (pbboche@sandia.gov).}
\and
Mitchell Luskin\footnotemark[2]
\and
Alexander V. Shapeev\footnotemark[2]
}

\date{\today}

\begin{document}
\maketitle
\begin{abstract}
We present a new optimization-based method for
atomistic-to-continuum (AtC) coupling. The main idea is to cast
the coupling of the atomistic and continuum models as a
constrained optimization problem with virtual Dirichlet
controls on the interfaces between the atomistic and continuum
subdomains. The optimization objective is to minimize the error
between the atomistic and continuum solutions on the overlap
between the two subdomains, while the atomistic and continuum
force balance equations provide the constraints.
Splitting of the atomistic and continuum problems instead
of blending them and their subsequent use as constraints in the
optimization problem distinguishes our approach from the
existing AtC formulations.
We present and analyze the method in the context of a
one-dimensional chain of atoms modeled using a linearized
two-body next-nearest neighbor interactions.
\end{abstract}

%%%%%%%%%%%%%%%%%%%%%%%%%%%%%%%%%%%%%%%%%%%%%%%%%%%%%%%%%%%%%%%%%%%%%%
\section{Introduction}
%%%%%%%%%%%%%%%%%%%%%%%%%%%%%%%%%%%%%%%%%%%%%%%%%%%%%%%%%%%%%%%%%%%%%%

Atomistic-to-continuum (AtC) coupling methods aim to combine the efficiency of continuum models such as PDEs with the accuracy of the atomistic models necessary to resolve local features such as cracks or dislocations that can affect the global material behavior.
Specifically, suppose that an atomistic model gives an accurate description of the true material behavior in a domain $\Omega$, but that this model is prohibitively expensive to solve on the whole domain.
The core of AtC formulations is to keep this model only where the
fully atomistic description is required to accurately represent
local features, while utilizing a more efficient continuum
model in the rest of $\Omega$.  Existing AtC methods differ chiefly
by the manner in which these models are joined together, which is
also the main challenge in the AtC formulation.

To explain the main ideas we can consider a scenario where
$\Omega$ is subdivided into an atomistic and  a continuum
subdomain, $\Omega_a$ and $\Omega_c$, such that $\Omega_a \cup
\Omega_c = \Omega$ and $\Omega_a \cap \Omega_c =: \Omega_o \neq
\emptyset$.  The overlap domain $\Omega_o$ is often referred to
as the handshake or blending region. In an AtC method, we
use the atomistic description on $\Omega_a$ and the continuum
description on $\Omega_c$. The problem then is how to couple
the two different descriptions of the material over the
overlap region, $\Omega_o$.

Attempts at this problem thus far can be characterized as
either energy-based, where a coupled energy is defined to be
minimized; or force-based, where internal and external forces
in $\Omega_a$ and $\Omega_c$ are equilibrated. In either
case, the resulting AtC methods often involve some form
of blending of the energy and/or the forces over the overlap
region.

The extension of the Arlequin method \cite{Bauman_08_CM}
and quasicontinuum method \cite{koten_2011,bqce12} are examples
of blended energy AtC methods in which the continuum and
atomistic energies are combined over $\Omega_o$ using a
partition of unity. The blended functional is then minimized
over $\Omega$ subject to a constraint expressing equality (in a
suitable sense) of the atomistic and continuum displacements in
$\Omega_o$. %Blended energy-based quasicontinuum methods have
%been proposed and studied in \cite{koten_2011,bqce12}.

A standard way to define an AtC method using force
blending is to start from the variational form of the atomistic
and continuum models and blend the corresponding weak forms
over $\Omega_o$. We refer to \cite{Bochev_08_MMS} and
\cite{bqcf} for investigations of blended force-based AtC
formulations. A more extensive investigation of these and
other AtC methods has been presented in \cite{acta.atc}.

In this paper, we formulate and analyze an
optimization-based AtC method which differs significantly from
the existing AtC approaches. The main idea is to cast AtC as a
constrained optimization problem with virtual Dirichlet
controls on the interfaces between the atomistic and continuum
subdomains. The objective in this optimization problem is to
minimize a suitable norm of the difference between the
atomistic and continuum displacement fields over the overlap
region $\Omega_o$, while the atomistic and continuum force
balance equations on $\Omega_a$ and $\Omega_c$ provide the
constraints.
Because we consider splitting the original problem over
$\Omega$ into two subproblems over $\Omega_a$ and $\Omega_c$,
we also must impose some form of boundary condition on the two
interfaces between $\Omega_a$ and $\Omega_c$ in order to have a
well-posed problem. In the context of the optimization
formulation, these boundary conditions act as Dirichlet
controls. However, they are virtual, or artificial controls,
because the boundaries on which they are imposed are artificial
rather than actual domain boundaries.

While the  AtC methods in
\cite{Bauman_08_CM,Chamoin_10_IJNME,bqcf,koten_2011,bqce12}
also involve a constrained optimization formulation, they
differ fundamentally from the approach developed in this paper.
Most notably, the former \emph{minimize a blended energy
functional} subject to constraints \emph{forcing the equality}
of the atomistic and continuum displacements over $\Omega_o$.
In contrast, our approach completely separates the two models
and minimizes the discrepancy of the atomistic and continuum
displacements in $\Omega_o$ subject to the two models acting independently in
$\Omega_a$ and $\Omega_c$.
The reversal of the roles of the constraints and objectives in
our approach, relative to blending methods, bears some
important theoretical and computational advantages. It
addresses the problem of blending two
physical models over a shared spatial region by
minimizing instead the
mismatch of the deformations in the overlap region which is
less restrictive on the overall formulation.

The use of optimization and control ideas for AtC further gives
rise to a number of attractive theoretical and computational properties.
For instance, we are able to infer
key properties of the method from its atomistic and
continuum constituencies, as illustrated in the proof of
Lemma~\ref{Q}.  This should be contrasted to the
force-based quasicontinuum operator which fails to be stable in specific
norms even though both atomistic and continuum force
operators are stable~\cite{DobsonLuskinOrtner2010b}.

The primary computational advantage of the proposed method
is that code to implement the method can be built upon preexisting code
for solving individual atomistic and continuum problems.
Since the core feature is minimizing the difference between solutions of
an atomistic model and a continuum model, all that is required is a linking
program between the two algorithms which carries out the optimization.

Our AtC work follows a number of previous efforts exploring the
use of optimization and control ideas for the design of numerical methods
\cite{Lions_00_JAM,Lions_00_CRASP,Bochev_09_SINUM,Bochev_09a_LSSC,Bochev_11a_LSSC}.
Conceptually, our approach is closest to the
virtual control techniques for heterogeneous domain decomposition developed in~\cite{gervasio_2001}.  In
that setting, both domains are modeled using a local, continuum (PDE)
model, whereas here we are concerned solely with coupling a
nonlocal atomistic model with a local continuum model.

Since the main goal of this paper is to demonstrate the
application of optimization and control ideas to AtC, we
formulate and analyze our method using a linearized Lennard-Jones
type atomistic
model~\cite{curtin_2003,parks_2008,doblusort:qce.stab}. For
completeness, we review this model in Section \ref{sec:prelim}.
We present the new method in Section \ref{S3} and analyze its
error in Section \ref{sec:error}. Section \ref{sec:concl}
summarizes our conclusions.

%%%%%%%%%%%%%%%%%%%%%%%%%%%%%%%%%%%%%%%%%%%%%%%%%%%%%%%%%%%%%%%%%%%%%%
\section{Preliminaries}\label{sec:prelim}
%%%%%%%%%%%%%%%%%%%%%%%%%%%%%%%%%%%%%%%%%%%%%%%%%%%%%%%%%%%%%%%%%%%%%%
This section establishes the notation and defines the
model atomistic problem. Application of the Cauchy-Born
rule~\cite{doblusort:qce.stab} to this problem yields the
continuum formulation.

\subsection{The atomistic model}
 We consider a chain of $N+1$ atoms with reference
(undeformed) positions $X_i$, $i=0,\ldots,N$. The atomic
positions in the deformed configuration are  $x_i$,
$i=0,\ldots,N$, and $u_i=x_i-X_i$ is the displacement of atom
$i$. We assume each atom interacts with its first and
second neighbors through a linearized Lennard-Jones type potential.  Thus, we can effectively think of
atoms interacting with first and
second neighbors via linear springs with spring constants
$k_1 > 0$ and $k_2$, respectively, and equilibrium lengths  $\ell$
and $2\ell$, respectively. %~\cite{curtin_2003}.
For the linearization of typical
interatomic potentials such as Lennard-Jones, the second neighbor spring constant satisfies
$k_2 < 0$ \cite{doblusort:qce.stab} and so is not a physical spring, but we will assume $k_2 < 0$ in the following. We also assume nearest
neighbor interactions dominate second neighbor interactions with the hypothesis
\begin{equation}\label{nCondition}
k_1 + 4k_2 > 0,
\end{equation}
which is also necessary and sufficient for the stability of the atomistic problem~\cite{DobsonLuskinOrtner2010b} we consider below. For simplicity, we set the lattice parameter $\ell = 1$.

Under these assumptions, the computational domain is $\Omega
:= [0,N] \cap \mathbb{Z}$, and the  lattice displacements
$u=\{u_0,\ldots,u_{N}\}$ are elements of the space
\[
\mathcal{U} := \left\{u: \Omega \to \mathbb{R}\right\}
\]
with inner product
$(\cdot,\cdot)_{\ell^2(\Omega)}$ and norm
$\|\cdot\|_{\ell^2(\Omega)} =
(\cdot,\cdot)_{\ell^2(\Omega)}^{1/2}.$
The left and right ``boundaries'' of $\Omega$ are  $\Gamma^{-}=\{0,1\}$ and $\Gamma^{+}=\{N-1,N\}$, respectively\footnote{The reason for fixing two boundary atoms is to ensure all unconstrained
atoms have a full set of neighbors to interact with and avoid
boundary defects},
and  $\Omega^\circ = [2,N-2] \cap \mathbb{Z}$ is the interior.
The size of a domain is $|\cdot|$, for example, $|\Omega|=(N+1)$ and $|\Omega^\circ|=N-3$.

The potential energy of the lattice is the sum of first and second neighbor interactions
\begin{equation}\label{intEnergy}
\mathcal{E}^a(u) := \sum_{i = 0}^{N-1}\frac{k_1}{2}(u_{i+1} - u_{i})^2
+\sum_{i = 1}^{N-1}\frac{k_2}{2}(u_{i+1} - u_{i-1})^2.
\end{equation}
We impose homogeneous Dirichlet boundary conditions by fixing
the atoms in $\Gamma^{-}$ and $\Gamma^{+}$. The corresponding
homogeneous space of admissible displacements is then
\[
\mathcal{U}_0 := \left\{u \in \mathcal{U} : \mbox{$u=0$ on $\Gamma^{-}\cup\Gamma^{+}$}\right\}.
\]

We assume that a dead load external force, $f \in \mathcal{U}_0$, is applied at each atom site resulting in a total energy of
\begin{equation}\label{totalEnergy}
\mathcal{E}^{tot}_a(u) = \mathcal{E}^a(u) - \left( f, u\right)_{\ell^2(\Omega)}.
\end{equation}
An equilibrium configuration of the lattice under the dead load is then given by
\begin{equation}\label{latMin}
\tilde{u}^a = \argmin_{u \in \mathcal{U}_0} \mathcal{E}^{tot}_a(u).
\end{equation}
The Euler-Lagrange equations
\begin{equation}\label{eq:EL}
\frac{\partial \mathcal{E}^{tot}_a(\tilde{u}^a)}{\partial u_i} = 0,\qquad i\in\Omega^\circ,
%3\le i \le N-2
\end{equation}
for  (\ref{latMin})  give the force balance constraints at
each internal atom. We express these constraints using the
finite difference  operators $\Delta_1,\Delta_2:
\mathcal{U}_0\rightarrow \mathcal{U}_0$ defined by
\begin{align*}
\left(\Delta_1u\right)_{i}       =~& u_{i-1} - 2u_i + u_{i+1}, \quad i\in\Omega^\circ,
%i = 3,\, \ldots,\, N-2,
\\
\left(\Delta_2u\right)_{i}       =~& u_{i-2} - 2u_i +u_{i+2}, \quad   i\in\Omega^\circ.
%i = 3,\, \ldots,\, N-2.
\end{align*}
From (\ref{eq:EL}), the internal force at site $i$ for $i\in\Omega^\circ$
%$3\lei \le N-2,$
equals $(k_1\Delta_1u  + k_2\Delta_2u)_i$.  Thus,
the necessary conditions for the equilibrium of the atomistic
system are
\begin{align}
-(k_1\Delta_1 \tilde{u}^a  + k_2\Delta_2 \tilde{u}^a)_i =~& f_i,  \quad i\in\Omega^\circ,
%i = 3, \ldots, N-2,
\label{testLabel0} \\
\tilde{u}^a_i =~& 0, \,  \quad i\in\Gamma^{-}\cup\Gamma^{+}  \label{testLabel01}.
\end{align}
The system of linear algebraic equations
(\ref{testLabel0})--(\ref{testLabel01}) represents the fully
atomistic problem, which we write compactly as:
\begin{equation}\label{eq:AP}
\mbox{find $\tilde{u}^a \in \mathcal{U}_0$ such that $A\tilde{u}^a=f$}\,,
\end{equation}
where $A := -k_1\Delta_1 - k_2\Delta_2$.

\subsection{The continuum model}
To derive the continuum (local) model, we use the
Cauchy-Born rule
$u_i\approx 1/2(u_{i+1}+u_{i-1})$; see~\cite{acta.atc},
to approximate the second
neighbor interactions by first neighbor interactions
\begin{equation}\label{cb}
(u_{i+1} - u_{i-1})^2\approx 2 (u_{i+1} - u_{i})^2
+2 (u_{i} - u_{i-1})^2.
\end{equation}
Substitution of the Cauchy-Born approximation \eqref{cb} into the atomistic
energy \eqref{intEnergy} yields the continuum potential energy
\begin{equation}\label{scb}
\mathcal{E}^c =
\frac{1}{2}\sum_{i=0}^{N-1}k_c(u_{i+1}-u_{i})^2
-k_2(u_1-u_0)^2-k_2(u_N-u_{N-1})^2  \,,
\end{equation}
where $k_c=k_1+4k_2.$ We note that a surface Cauchy-Born
correction is not needed for~\eqref{scb} since we are assuming
that $u_i=0$ for $i\in\Gamma^{-}\cup\Gamma^{-}$.

We now define the total continuum energy under a force $f\in
\mathcal{U}_0$ as
\begin{equation}\label{totalEnergyCont}
\mathcal{E}^{tot}_c(u) = \mathcal{E}^c(u) - \left( f, u\right)_{\ell^2(\Omega)}.
\end{equation}
An equilibrium configuration of the continuum model minimizes the total energy:
\begin{equation}\label{contMin}
\tilde u^c = \argmin_{u \in \mathcal{U}_0} \mathcal{E}^{tot}_c(u).
\end{equation}
The Euler-Lagrange equations for (\ref{contMin}) are given by
\begin{align}
-\left(k_c\Delta_1 \tilde u^c\right)_i =~& f_i,  \quad i\in\Omega^\circ,
%i = 3, \ldots, N-2,
\label{contLabel0} \\
\tilde u_i^c =~& 0, \,  \quad i\in\Gamma^{-}\cup\Gamma^{+}
%i = 1,\,2,\,N-1,\,N
\label{contLabel01}.
\end{align}
The system of linear algebraic equations
(\ref{contLabel0})--(\ref{contLabel01}) represents the
continuum problem. Setting $C =- k_c\Delta_1,$ this system
assumes the form:
\begin{equation}\label{eq:CP}
\mbox{find $\tilde u^c\in \mathcal{U}_0$ such that $C\tilde u^c=f$}\,.
\end{equation}

\begin{subsection}{The continuum modeling error}
The error in the approximation of the atomistic solution by the
 continuum solution is given by the following proposition.
\begin{prop}\label{cont.error}
There exists a fixed constant $c_0,$ independent of $N$ and
$\tilde{u}^a,$ such that
\begin{equation}\label{CBounda}
\| \tilde{u}^a - \tilde u^c\| _{\ell^2(\Omega)}
\leq
c_0N^2\| \Delta_1^2\tilde{u}^a\| _{\ell^2(\Omega)}.
\end{equation}
\end{prop}
\begin{proof}
To estimate the continuum modeling error, we observe that
$$
A - C = -\left(k_1\Delta_1 + k_2\Delta_2\right) + (k_1 + 4k_2)\Delta_1
      = -k_2(\Delta_2 - 4\Delta_1)
      = -k_2\Delta_1^2.
$$
%\begin{align*}
%A - C =~& -\left(k_1\Delta_1 + k_2\Delta_2\right) + (k_1 + 4k_2)\Delta_1 \\
%      =~& -k_2(\Delta_2 - 4\Delta_1) \\
%      =~& -k_2\Delta_1^2.
%\end{align*}
Now $C$ is just the $1D$ discrete Laplacian on $\Omega$ with
homogeneous Dirichlet boundary conditions at atom sites $1$ and $N-1.$ So,
the minimum eigenvalue for $C$ is $\lambda_1 =
4k_c\sin^2\left(\frac{\pi}{2(n+1)}\right)$ where $n = N-3$ is
the number of unconstrained atoms.
Using that  $\tilde{u}^a-\tilde u^c=0$ at
atoms $1$ and $N-1$ implies $C\tilde u^c = f = A\tilde{u}^a$, which yields the following bound:
\begin{equation*}
\begin{split}
\| \tilde{u}^a - \tilde u^c\| _{\ell^2(\Omega)}
\leq~& \| C^{-1}\| _{\ell^2(\Omega)}\cdot\| C\left(\tilde{u}^a - \tilde u^c\right)\| _{\ell^2(\Omega)}   \\
=~& \| C^{-1}\| _{\ell^2(\Omega)}\cdot\| C\tilde{u}^a - A\tilde{u}^a\| _{\ell^2(\Omega)} \\
=~& |k_2|\, \| C^{-1}\| _{\ell^2(\Omega)}\cdot\| \Delta_1^2\tilde{u}^a\| _{\ell^2(\Omega)} \\
=~& \frac{|k_2|}{4k_c\sin^2\left(\frac{\pi}{2(n+1)}\right)}\| \Delta_1^2\tilde{u}^a\| _{\ell^2(\Omega)}
%\\
\leq~
%&
c_0N^2\| \Delta_1^2\tilde{u}^a\| _{\ell^2(\Omega)}.
\end{split}
\end{equation*}
\end{proof}
\end{subsection}

%%%%%%%%%%%%%%%%%%%%%%%%%%%%%%%%%%%%%%%%%%%%%%%%%%%%%%%%%%%%%%%%%%%%%%
\section{Optimization-based AtC formulation}\label{S3}
%%%%%%%%%%%%%%%%%%%%%%%%%%%%%%%%%%%%%%%%%%%%%%%%%%%%%%%%%%%%%%%%%%%%%%
As with any AtC formulation, we begin by splitting
$\Omega$ into atomistic, continuum, and overlap regions
\[
\Omega_a = [0, L] \cap \mathbb{Z}, \quad \Omega_c = [K, N] \cap \mathbb{Z}, \quad \Omega_o = \Omega_a \cap \Omega_c = [K,L] \cap \mathbb{Z},
\]
where $0<K<L<N$. The strict interiors of these domains are
\[
\Omega_a^\circ = [2, L-2] \cap \mathbb{Z}, \quad \Omega_c^\circ = [K+2, N-2]
\cap \mathbb{Z}, \quad \Omega_o^\circ = \Omega_a^\circ \cap \Omega_c^\circ = [K+2,L-2] \cap \mathbb{Z},
\]
and their boundaries are
\begin{gather*}
\Gamma^{-}_a = \{0,1\}\quad\text{and}\quad\Gamma^{+}_a =
\{L-1,L\},\\
 \Gamma^{-}_c = \{K,K+1\}\quad\text{and}\quad\Gamma^{+}_c =
\{N-1,N\},\\
\Gamma^{-}_o = \{K,K+1\}\quad\text{and}\quad\Gamma^{+}_o =
\{L-1,L\}.
\end{gather*}
Here and in the remainder of the paper, we
find it convenient to use modified Vinogradov notation where
the implied constant is independent of the parameters $K,L,$
and $N$.  Thus, $X \gtrsim~ Y$ means there is a positive constant $c$
such that $X \geq c Y$ with $c$ independent of $K,L,$ and $N$.

Recall that the main objective of an AtC method is a
stable, accurate, and efficient approximation of the lattice
displacements by using the atomistic model on $\Omega_a$, employing the
continuum approximation on $\Omega_c$, and accurately
merging them together on $\Omega_o$. It follows that the
efficiency of AtC methods hinges on the assumption that the
atomistic region is small compared to the continuum region. On
the other hand, it is intuitively clear that a stable and
accurate AtC method requires some conditions on the size of
$\Omega_o$. These assumptions are pivotal to our analysis and
we formalize them below.
\begin{assumption}\label{assumpA}
There exists a real number $p > 1$ such that
$$
L\lesssim N^{1/p}.
$$
\end{assumption}
\begin{assumption}\label{assumpB}
There exists a real number $\gamma$, $\frac{3}{L} < \gamma < 1,$ such that
$$
\frac{L-K}{L} = \gamma.
$$
\end{assumption}
In other words, we assume that  $|\Omega_a | \lesssim
|\Omega|^{1/p}$ and  $|\Omega_o| = \gamma|\Omega_a|$ is
such that  $|\Omega_o| > 3$, i.e., the overlap region's size is at least twice
the size of the interaction range; see Fig.~\ref{fig:atoms}.  The assumption that $\gamma$ is constant means the ratio of the overlap width to the size of the atomistic region is constant, or equivalently, that the ratio of $K$ to $L$ is constant.
\begin{figure}[htp!]
\centering
\includegraphics{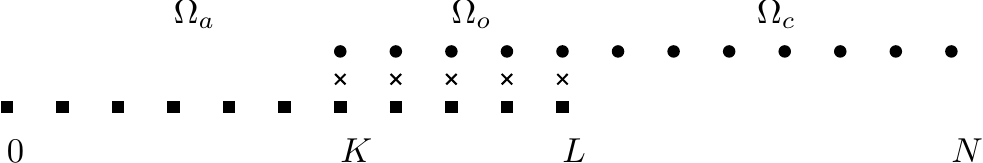}
\caption{Decomposition of $\Omega$.  Squares are in the atomistic region $\Omega_a$, circles are in the continuum region $\Omega_c$, and crosses are in the overlap region $\Omega_o$.}
\label{fig:atoms}
\end{figure}

We formulate the optimization-based AtC method in two
steps. The first step defines independent atomistic and
continuum subproblems on  $\Omega_a$ and $\Omega_c$,
respectively, whereas the second step merges these problems by
minimizing the mismatch of their solutions on $\Omega_o$. To
describe the first step, we introduce the spaces
\begin{align*}
\mathcal{U}_a :=~& \left\{u: \Omega_a \to \mathbb{R} \ | \ \mbox{$u=0$ on $\Gamma_a^{-}$}
%u_1 = u_2 = 0
\right\},    \\
\mathcal{U}_{a,0} :=~& \left\{u: \Omega_a \to \mathbb{R} \ | \ \mbox{$u=0$ on $\Gamma_a^{-}\cup\Gamma_a^{+}$}
%u_1 = u_2 = u_{L-1} = u_{L} = 0
\right\},    \\
\mathcal{U}_c :=~& \left\{u: \Omega_c \to \mathbb{R} \ | \ \mbox{$u=0$ on $\Gamma_c^{+}$}
%u_{N-1} = u_N = 0
\right\}, \\
\mathcal{U}_{c,0} :=~& \left\{u: \Omega_c \to \mathbb{R} \ | \ \mbox{$u=0$ on $\Gamma_c^{-}\cup\Gamma_c^{+}$}
%u_{K} = u_{K+1} = u_{N-1} = u_N = 0
\right\},
\end{align*}
for the subdomain displacements and the ``trace'' spaces
\begin{equation}\label{eq:trace}
\Lambda_a = \left\{w:\Gamma_a^+ \to \mathbb{R}\right\}\quad\mbox{and}\quad
\Lambda_c = \left\{w:\Gamma_c^{-} \to \mathbb{R}\right\}
\end{equation}
for the displacement values on the artificial domain boundaries $\Gamma^{+}_a$ and $\Gamma^{-}_c$.
We denote the standard $\ell^2$ inner product
and norm on these spaces by $(\cdot,\cdot)_{\ell^2(\sigma)}$
and $\|\cdot\|_{\ell^2(\sigma)}$, where $\sigma$ stands for the
appropriate domain under consideration.
The trace spaces provide the boundary conditions on
$\Gamma^{+}_a$ and $\Gamma^{-}_c$ necessary to formulate
well-posed atomistic and continuum problems on $\Omega_a$ and
$\Omega_c$.

Let $A_a$ and $f^a$ be the restrictions of $A$ and $f$ to the interior $\Omega_a^\circ$. Likewise, let
$C_c$ and $f^c$ denote the restrictions of $C$ and $f$ to the interior $\Omega_c^\circ$.  The local nature of the continuum subdomain operator $C_c$ necessitates the need for only a single boundary constraint on $\Gamma^{-}_c$ whereas the atomistic subdomain operator $A_a$ requires two constraints on $\Gamma^{+}_a$.  For this reason, when referring to the continuum model, we adjust the definitions of $\Omega_c$, $\Omega_c^\circ$, $\Gamma^{-}_c$, and $\Gamma^{+}_c$ to be
\[
\Omega_c = [K, N-1] \cap \mathbb{Z},\quad
\Omega_c^\circ = [K+1, N-2] \cap \mathbb{Z}, \quad \Gamma^{-}_c = \{K\}, \quad \Gamma^{+}_c = \{N-1\},
\]
with analogous changes made to overlap boundaries and interiors
and the displacement and trace spaces.  Thus, the left,
artificial boundary and the right, true boundary of the continuum region are single atoms.  The continuum
displacement at atom $N$ is zero since this is a true boundary condition of the original problem.  To simplify notation in upcoming computations, we set $\bar{N} := N-1$; see~\eqref{fig:boundary}.
\begin{figure}[htp!]
\centering
\includegraphics{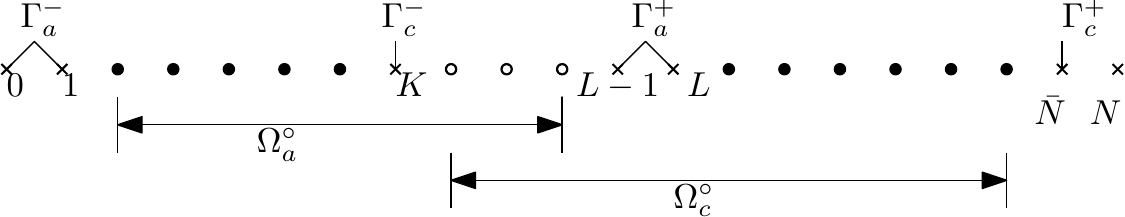}
\caption{Trace spaces and interiors of $\Omega_a,\, \Omega_c$.  The interior of $\Omega_o$ is depicted with open circles.}
\label{fig:boundary}
\end{figure}

We define the atomistic subproblem as a restriction of
(\ref{eq:AP}) to $\Omega_a$ with inhomogeneous boundary
conditions at the artificial atomistic boundary $\Gamma^{+}_a$,
i.e., given $\theta^a\in\Lambda^a$ we seek $u^a \in
\mathcal{U}_a$ such that
\begin{equation}\label{eq:subatom}
\left\{
\begin{array}{rl}
A_au^a = f^a    & \mbox{on $\Omega_a^\circ$} \\[1ex]
u^a       = \theta^a   &  \mbox{on $\Gamma_a^{+}$}
\end{array}
\right. .
\end{equation}
Similarly, the continuum subproblem is a restriction of (\ref{eq:CP}) to $\Omega_c$ with inhomogeneous boundary condition at the artificial continuum boundary $\Gamma^{-}_c$: given $\theta^c \in \Lambda_c$ we seek
$u^c \in \mathcal{U}_c$ such that
\begin{equation}\label{eq:subcont}
\left\{
\begin{array}{rl}
C_cu^c = f^c    & \mbox{on $\Omega_c^\circ$} \\[1ex]
u^c       = \theta^c   &  \mbox{on $\Gamma_c^{-}$}
\end{array}
\right. .
\end{equation}
Thanks to the boundary conditions prescribed on the artificial
boundaries, the subdomain problems
(\ref{eq:subatom})--(\ref{eq:subcont}) are well-posed and can
be solved for any given $\theta^a$ and $\theta^c$.  However,
because $\theta^a$ and $\theta^c$ are unknown, the solutions to
(\ref{eq:subatom}) and (\ref{eq:subcont}) cannot yet be
determined.

The second step in the formulation of our AtC method is
the merging of (\ref{eq:subatom}) and (\ref{eq:subcont})
into a single well-posed problem for the unknown states $u^a$ and
$u^c,$ and the unknown boundary conditions $\theta^a$ and
$\theta^c$.  Intuitively, we desire that
\begin{equation}\label{eq:equality}
\theta^c\approx u^a\ \mbox{on $\Gamma^{-}_c$},\quad
\theta^a\approx u^c\ \mbox{on $\Gamma^{+}_a$}, \quad\mbox{and}\quad
u^a\approx u^c\ \mbox{in $\Omega^\circ_o$}.
\end{equation}
In many hybrid AtC methods these, or similar conditions, are
used to constrain the hybrid force balance equations or the
minimization of a hybrid energy functional; see e.g.,
\cite{Bauman_08_CM,Chamoin_10_IJNME,bqcf}.
However, there is no canonical way of enforcing strong or weak type equality of fundamentally different atomistic and continuum solution states.

The cornerstone of our optimization-based AtC approach is
to view (\ref{eq:equality}) as the optimization objective
rather than as the constraint. Specifically, in the context of
our model problem, the quantity
\begin{equation}\label{eq:mismatch}
\|u^a-u^c\|^2_{\ell^2(\Omega_o)}=~ \|u^a -
\theta^c\|^2_{\ell^2(\Lambda_c)}+ \|u^a -
u^c\|^2_{\ell^2(\Omega^\circ_o)}+ \|u^c -
\theta^a\|^2_{\ell^2(\Lambda_a)}
\end{equation}
provides a notion of an artificial ``mismatch'' energy between
the solutions of (\ref{eq:subatom}) and (\ref{eq:subcont})  in
the overlap region. Instead of forcing this energy to be
exactly zero, which does not yield a problem with a
solution, we seek to minimize it subject to the atomistic and
continuum force balance equations (\ref{eq:subatom}) and
(\ref{eq:subcont}) holding independently in $\Omega_a$ and
$\Omega_c$. Succinctly, our new AtC formulation is the
following constrained optimization problem:
\begin{equation}\label{opt1}
\min_{\{u^a,u^c,\theta^a,\theta^c\}} \frac{1}{2}\| u^a - u^c\| _{\ell_2(\Omega_o)}^2
\ \mbox{s.t.}\
\left\{\!\!\!
\begin{array}{rl}
A_au^a = f^a    & \!\! \mbox{on $\Omega_a^\circ$} \\[1ex]
u^a       = \theta^a   & \!\! \mbox{on $\Gamma_a^{+}$}
\end{array},
\right.
\begin{array}{rl}
C_cu^c = f^c    &\!\! \mbox{on $\Omega_c^\circ$} \\[1ex]
u^c       = \theta^c   &\!\!  \mbox{on $\Gamma_c^{-}$}
\end{array}.
\end{equation}
In the language of constrained optimization, the
functions $u^a\in\mathcal{U}_a$ and $u^c\in\mathcal{U}_c$ are
the \emph{states}, and the artificial boundary conditions $\theta^a\in\Lambda_a$ and
$\theta^c\in\Lambda_c$  are the \emph{controls}. The purpose of
the controls is to allow the states to adjust so as to provide
the smallest possible value of the objective while still
satisfying the constraints. In the context of (\ref{opt1}),
$\theta^a$ and $\theta^c$ are \emph{virtual boundary} controls,
as the boundaries $\Gamma^+_a$ and $\Gamma^{-}_c$ are an
artifact of the domain decomposition into atomistic and
continuum parts.

We shall show below that the optimization
problem~\eqref{opt1} is well-posed. But before investigating
this, we show the optimization-based AtC
formulation~\eqref{opt1} satisfies a patch test criterion.

\subsection{Patch Test Consistency}
The bane of all atomistic-to-continuum coupling mechanisms is
the existence of nonphysical ghost forces arising on the
interface of the continuum and atomistic
regions~\cite{Shenoy:1999a,dobsonluskin09}.  The patch test is
a well-known test for determining the existence of ghost forces
by checking whether a uniform strain is an equilibrium solution
to the proposed method on a perfect lattice under zero external
forces~\cite{Shenoy:1999a,acta.atc}.  As with force-based
methods, the optimization formulation~\eqref{opt1} is patch
test consistent by design. Indeed, if we replace the
homogeneous Dirichlet boundary conditions by the inhomogeneous
boundary conditions
\[
u^a_{0} = 0, \ u^a_1 = F \ \mbox{ and } \ u^c_{N-1} = (N-1)F, \ u^c_{N} = NF,
\]
where $F > 0$ defines a macroscopic displacement gradient, and if we take take $f^a \equiv 0, f^c \equiv 0$, then it is
straightforward to verify that~\eqref{opt1} has a minimum of
$0$ achieved when $u^{a}_{i} = iF$ and $u^c_{i} = iF$.  This is
due to the fact that both atomistic and continuum operators are
patch test consistent individually.

%%%%%%%%%%%%%%%%%%%%%%%%%%%%%%%%%%%%%%%%%%%%%%%%%%
\subsection{Well-Posedness}
%%%%%%%%%%%%%%%%%%%%%%%%%%%%%%%%%%%%%%%%%%%%%%%%%%
To establish that the optimization-based AtC
formulation~\eqref{opt1} is well-posed, we switch to the
reduced space form of the optimization problem, which requires
the elimination of the states from~\eqref{opt1}. In our case,
this task is trivial because for any pair of virtual controls
$\{\theta^a,\theta^c\}\in\Lambda_a\times\Lambda_c$ the
constraints
\begin{equation}\label{eq:subprob}
\left\{
\begin{array}{rl}
              A_au^a  = f^a         & \mbox{on $\Omega_a^\circ$},\\[1.5ex]
                    u^a  = \theta^a & \mbox{on $\Gamma^{+}_a,$}
\end{array}
\right.
\quad\mbox{and}\quad
\left\{
\begin{array}{rl}
              C_cu^c      = f^c                  & \mbox{on $\Omega_c^\circ$},\\[1.25ex]
                     u^c = \theta^c     &\mbox{on $\Gamma^{-}_c,$}
\end{array}
\right.
\end{equation}
have unique solutions $u^a = u^a\left(\theta^a\right)\in \mathcal{U}_a$ and  $u^c  = u^c\left(\theta^c\right) \in \mathcal{U}_c$.
Using these solutions in~\eqref{opt1} transforms the latter into an equivalent unconstrained minimization problem
\begin{equation}\label{reduced}
\min\limits_{\left\{\theta^a, \, \theta^c\right\} \in \Lambda_a \times \Lambda_c}
\frac{1}{2}\| u^a(\theta^a) - u^c(\theta^c)\| ^2_{\ell^2(\Omega_o)} \,.
\end{equation}
This problem, in terms of the virtual controls only, is the reduced space form of ~\eqref{opt1}.
We analyze (\ref{reduced}) following the strategy in Gervasio et al.~\cite{gervasio_2001}.  Specifically,
for any given $\left\{\theta^a, \, \theta^c\right\} \in \Lambda_a \times \Lambda_c$
we split the solutions of the constraint equations (\ref{eq:subprob}) as
\begin{equation}\label{decomp}
u^{a}(\theta^{a}) = v^{a}(\theta^{a}) + u^{a,0}
\quad\mbox{and}\quad
u^{c}(\theta^{c}) = v^{c}(\theta^{c}) + u^{c,0}
\end{equation}
where the homogeneous components $u^{a,0}\in\mathcal{U}_{a,0}$ and $u^{c,0}\in\mathcal{U}_{c,0}$ solve
\begin{equation}\label{splitting1}
A_au^{a,0} =~ f^a\quad\mbox{and}\quad C_cu^{c,0} = f^c,
\end{equation}
respectively, whereas $v^{a}(\theta^{a})\in\mathcal{U}_a$ and $v^{c}(\theta^{c})\in\mathcal{U}_c$ solve
\begin{equation}\label{prob1}
\left\{
\begin{array}{rl}
              A_av^a  = 0         & \mbox{on $\Omega_a^\circ$}\\[1.25ex]
                    v^a  = \theta^a & \mbox{on $\Gamma^{+}_a$}
\end{array}
\right.
\quad\mbox{and}\quad
\left\{
\begin{array}{rl}
              C_cv^c      = 0             & \mbox{on $\Omega_c^\circ$}\\[1.25ex]
                     v^c = \theta^c     &\mbox{on $\Gamma^{-}_c$}
\end{array}
\right.,
\end{equation}
respectively.  We will prove the following stability result for
\eqref{prob1} in Appendix~\ref{A}.
\begin{lemma}\label{stab}
%There exists a constant $c$, independent of $K,L,$ and $N$, such that
For any $\left\{\theta^a, \, \theta^c\right\} \in \Lambda_a \times \Lambda_c$, the solutions
$v^a(\theta^a)$ and $v^c(\theta^c)$ to \eqref{prob1} satisfy the bounds
\begin{equation}\label{stabbound}
\begin{split}
\| v^a(\theta^a)\| ^2_{\ell^2(\Omega_a)}&\lesssim~L   \|\theta^a\|^2_{\ell^2(\Gamma^+_a)} ,\\
%cL\left( (\theta^a_L)^2+(\theta^a_{L-1})^2 \right),\\
\| v^c(\theta^c)\| ^2_{\ell^2(\Omega_c)}&\le (N-K)  \|\theta^c\|^2_{\ell^2(\Gamma^-_c)}
%(\theta^c_K)^2.
\end{split}
\end{equation}
\end{lemma}
%%%%%%

\begin{rem}\label{rem:bounds}
The bounds (\ref{stabbound}) continue to hold when we take homogeneous Dirichlet boundary conditions on $\Gamma_a^+$ and $\Gamma_c^{-}$ in (\ref{prob1}) and inhomogeneous boundary conditions on $\Gamma_a^-$ and $\Gamma_c^+$ by replacing $\Gamma_a^+$ with $\Gamma_a^-$ and $\Gamma^-_c$ with $\Gamma^+_c$.
\end{rem}

Using the decomposition~\eqref{decomp}, the reduced space
problem (\ref{reduced}) assumes the form
\begin{equation}\label{unconstrained}
\begin{split}
\min\limits_{\left\{\theta^a, \, \theta^c\right\} \in \Lambda_a \times \Lambda_c} \frac{1}{2}\| v^a(\theta^a) - v^c(\theta^c)\| ^2_{\ell^2(\Omega_o)} &+ \left(v^a(\theta^a) - v^c(\theta^c),u^{a,0}-u^{c,0}\right)_{\ell^2(\Omega_o)}  \\
&\quad + \frac{1}{2}\| u^{a,0} - u^{c,0}\| ^2_{\ell^2(\Omega_o)}.
\end{split}
\end{equation}
The following result is key to proving that the reduced space problem  (\ref{reduced}), respectively (\ref{unconstrained}), has a unique minimizer.
\begin{thm}\label{norm1}
The form
\begin{equation}\label{eq:bform}
\left\langle \left\{\theta^a, \theta^c\right\}, \left\{\mu^a, \mu^c\right\}\right\rangle ~:= \left(v^a(\theta^a) - v^c(\theta^c), v^a(\mu^a) - v^c(\mu^c)\right)_{\ell^2(\Omega_o)}
\end{equation}
defines an inner product on $\Lambda_a \times \Lambda_c$.
\end{thm}
\smallskip

\begin{proof} The proof follows from the inequality (\ref{qEstimate22a}) in Lemma~\ref{Q}.

\end{proof}

\smallskip

Theorem \ref{norm1} allows us to recast the reduced space problem (\ref{unconstrained}) as
\begin{equation}\label{woot}
\begin{split}
\min\limits_{\left\{\theta^a, \theta^c\right\} \in \Lambda_a \times
\Lambda_c}
\frac{1}{2} \| \{ \theta^a,\theta^c \}  \|^2_{\ell^\star(\Lambda_a\times\Lambda_c)}
%\|\{\theta^a, \theta^c\right\}\|^2}
%_{l^*(\Lambda_a\times\Lambda_c)}}
%\frac{1}{2}\langle\left\{\theta^a, \theta^c\right\},\left\{\theta^a, \theta^c\right\}\rangle
&+ \left(v^a(\theta^a) - v^c(\theta^c),u^{a,0}-u^{c,0}\right)_{\ell^2(\Omega_o)} \\
&\quad \quad + \frac{1}{2}\| u^{a,0} - u^{c,0}\| ^2_{\ell^2(\Omega_o)} ,
\end{split}
\end{equation}
where $\|\cdot\|_{\ell^\star(\Lambda_a\times\Lambda_c)}$ is the norm induced by (\ref{eq:bform}).
The necessary optimality condition (Euler-Lagrange equation) for (\ref{woot}) is the following variational equation: find
$\left\{\theta^a, \theta^c\right\}\in\Lambda_a\times\Lambda_c$ such that
\begin{equation}\label{varForm}
\left< \left\{\theta^a, \theta^c\right\}, \left\{\mu^a, \mu^c\right\}\right> ~=~-\left(u^{a,0}-u^{c,0},  v^a(\mu^a) - v^c(\mu^c)\right)_{\ell^2(\Omega_o)}
%F\left(\left\{\mu^a, \mu^c\right\}\right)
%\ \mbox{for all} \ \left\{\mu^a, \mu^c\right\} \in \left(\Lambda_a \times \Lambda_c\right).
\end{equation}
for all $\left\{\mu^a, \mu^c\right\} \in \left(\Lambda_a \times
\Lambda_c\right).$
Theorem~\ref{norm1} and the Riesz
representation theorem imply
that (\ref{varForm}) has a unique solution, thereby establishing the well-posedness of the reduced space problem
~\eqref{unconstrained}.

\begin{rem}\label{rem1}
There are two principal pathways for the analysis and the numerical solution of the constrained optimization problem
(\ref{opt1}). The first one, which we adopt in this paper, relies on the strictly convex reduced space problem (\ref{reduced}) and its equivalent forms (\ref{unconstrained}) and
(\ref{woot}). In this case the practical implementation of the AtC method involves the solution of the strongly coercive Euler-Lagrange equation (\ref{varForm}), followed by the recovery of the atomistic and continuum states from the virtual controls.

The second pathway relies on Lagrange multipliers to enforce the constraints in (\ref{opt1}) and yields a saddle-point optimization problem. The latter can be analyzed using the Brezzi's theory~\cite{brezzi_1974}, while implementation of the AtC method then requires the solution of a weakly-coercive, mixed-type optimality system. We note that showing the conditions of the Brezzi theory is essentially equivalent to showing the well-posedness of the reduced space problem.
\end{rem}

\begin{rem}\label{rem2}
Let $A_c$ be the restriction of $A$ to the continuum subdomain $\Omega_c$. Substituting $A_c$ for $C_c$ in
(\ref{opt1}) yields a constrained optimization formulation that is equivalent to the global atomistic problem
(\ref{eq:AP}). The corresponding reduced space problem and its Euler-Lagrange equation differ from (\ref{woot}) and
(\ref{varForm}) only by the use of the atomistic operator $A_c$ instead of $C_c$.
As a result, the variational problem (\ref{varForm}) can be thought of as resulting from a modification of the bilinear form associated to the original problem very similar to a nonconforming finite element method.
\end{rem}

%%%%%%%%%%%%%%%%%%%%%%%%%%%%%%%%%%%%%%%%%%%%%%%%%%%%%%%%%%%%%%%%%%%%%%
\section{Consistency and Error Analysis}\label{sec:error}
%%%%%%%%%%%%%%%%%%%%%%%%%%%%%%%%%%%%%%%%%%%%%%%%%%%%%%%%%%%%%%%%%%%%%%
This section analyzes the error between the true solution
$\tilde{u}^a$ to the atomistic problem~\eqref{testLabel0} and
the solution of the optimization-based AtC formulation
(\ref{opt1}).
%Here and in the remainder of the paper, we
%find it convenient to use modified Vinogradov notation where
%the implied constant is independent of the parameters $K,L,$
%and $N$.  Thus, $X \gtrsim~ Y$ means there is a constant $c$
%such that $X \geq c Y$ with $c$ independent of $K,L,$ and $N$.
%
%For a thorough asymptotic analysis, we also make the following
%standing assumption relating $K,L$, and $N$, which is
%reasonable under the heuristic that the atomistic model in an
%AtC coupling method is used only on a small portion of the
%domain.
%\begin{assumption}\label{assumpMain}
%There exists $0 < \gamma < 1$ such that
%\[
%K = \gamma L.
%\]
%Moreover, there exists $p > 1$ such that
%\[
%N^{1/p} \gtrsim L.
%\]
%\end{assumption}
%

To proceed with our analysis, let  $\{{\theta}^a_{op},{\theta}^c_{op}\}\in\Lambda_a\times\Lambda_c$ denote the
optimal solution of the reduced space problem (\ref{woot}) or,
what is the same---the solution of the Euler-Lagrange
equation (\ref{varForm}). The optimal solution of the full
problem (\ref{opt1}) is then given by
$\{{u}^a_{op},u^c_{op},\theta^a_{op},\theta^c_{op}\}$, where
\[
u^{a}_{op} = v^{a}(\theta^{a}_{op}) + u^{a,0}
\quad\mbox{and}\quad
u^{c}_{op} = v^{c}(\theta^{c}_{op}) + u^{c,0}.
\]
These are the optimal states. Using these states we define the AtC approximation to  $\tilde{u}^a$ as
\begin{equation}\label{eq:atcapprox}
u^{atc} :=
\begin{cases}
u^a_{op} & \mbox{in $\Omega_a$}, \\
u^c_{op} & \mbox{in $\Omega_c\backslash\Omega_o$}\,.
\end{cases}
\end{equation}
For the error analysis, it is convenient to express the
approximate AtC solution as
$$
u^{atc} = P\left(\{\theta^{a}_{op},\theta^{c}_{op}\}\right)\,,
$$
where the affine operator, $P: \Lambda_a \times \Lambda_c \to \mathcal{U}_0,$ is defined by
\begin{equation}\label{projectionDef}
P\left(\left\{\mu^a, \mu^c\right\}\right) :=  \begin{cases} & u^{a,0} + v^a(\mu^a) \ \mbox{in} \  \Omega_a, \\
                                                            & u^{c,0} + v^c(\mu^c) \ \mbox{in} \ \Omega_c\backslash\Omega_o,                                                               \end{cases}
\quad \forall \, \{\mu^a,\mu^c\}\in\Lambda_a\times\Lambda_c\,.
\end{equation}
Thus, the error of the AtC approximation (\ref{eq:atcapprox}) is
\begin{equation}\label{eq:atcerror}
\| \tilde{u}^a - u^{atc} \|_{\ell^2(\Omega)} =
\| \tilde{u}^a - P\left(\left\{\theta^a_{op}, \theta^c_{op}\right\}\right) \|_{\ell^2(\Omega)}\,.
\end{equation}
%
%where we have defined the coupled approximation, $u^{atc}$, by $u^{atc} := P\left(\left\{\theta^a_{op}, \theta^c_{op}\right\}\right)$.

To analyze (\ref{eq:atcerror}) it is advantageous to split $P$ into a
linear part, $Q$, and constant term, $U^0$, dependent only on the homogeneous data, i.e., $P=Q + U^0$ where
\begin{gather*}
Q\left(\left\{\mu^a, \mu^c\right\}\right) :=~  \begin{cases}  v^a(\mu^a) \ \mbox{in} \  \Omega_a, \\
                                                              v^c(\mu^c) \ \mbox{in} \
                                                              \Omega_c\backslash\Omega_o,
                                                               \end{cases}\quad\mbox{and}\quad
                                         U^{0} :=~ \begin{cases}  u^{a,0} \ \mbox{in} \  \Omega_a, \\
                                                                u^{c,0} \ \mbox{in} \ \Omega_c\backslash\Omega_o \,.
                                                               \end{cases}
\end{gather*}
We also introduce the trace operator $r: \mathcal{U}_0
\mapsto \Lambda_a \times \Lambda_c$ such that
\begin{equation}\label{representationDef}
r(u) := \left\{ \begin{pmatrix} u_{L-1} \\ u_{L} \end{pmatrix},\begin{pmatrix} u_{K}  \end{pmatrix} \right\} := \left\{ r^a(u), r^c(u)\right\}\quad\forall u \in \mathcal{U}_0\,.
\end{equation}
Because $r^a(\tilde{u}^a)$ contains the exact values of the atomistic solution, it follows that
\begin{equation}\label{ua}
\tilde{u}^a |_{\Omega_a}=v^a\left(r^a(\tilde{u}^a)\right) + u^{a,0} \,.
\end{equation}
We define the continuum lifting of the exact atomistic trace on $\Lambda_c$ as
\begin{equation}\label{uc}
u^c:=v^c\left(r^c(\tilde{u}^a)\right) + u^{c,0} \,.
\end{equation}
It is a matter of unraveling these definitions to see that
\begin{equation}\label{projectionrestrict}
P(r(\tilde{u}^a))=P\left(\left\{r^a(\tilde{u}^a), r^c(\tilde{u}^a)\right\}\right) =
\begin{cases}  \tilde{u}^a & \mbox{in} \  \Omega_a, \\
                      u^c & \mbox{in} \ \Omega_c\backslash\Omega_o                                                            \end{cases}\,.
\end{equation}
To estimate the AtC approximation error we split (\ref{eq:atcerror}) into two parts:
\begin{equation}\label{est0}
\begin{split}
\| \tilde{u}^a - P&\left(\left\{\theta^a_{op}, \theta^c_{op}\right\}\right)\| _{\ell^2(\Omega)}  \\
=~& \| \tilde{u}^a - P\left(r(\tilde{u}^a)\right) + P\left(r(\tilde{u}^a)\right) -
    P\left(\left\{\theta^a_{op}, \theta^c_{op}\right\}\right)\| _{\ell^2(\Omega)}  \\
  \leq~& \| \tilde{u}^a - P(r(\tilde{u}^a))\| _{\ell^2(\Omega)} + \| Qr(\tilde{u}^a) + U^0 -Q\left\{\theta^a_{op}, \theta^c_{op}\right\} - U^0\| _{\ell^2(\Omega)} \\
  =~& \| \tilde{u}^a - P(r(\tilde{u}^a))\| _{\ell^2(\Omega)} + \| Q\left(r(\tilde{u}^a) - \left\{\theta^a_{op}, \theta^c_{op}\right\}\right)\| _{\ell^2(\Omega)} \\
\leq~& \| \tilde{u}^a - P(r(\tilde{u}^a))\| _{\ell^2(\Omega)} +
   \| Q\| \cdot
   \| r(\tilde{u}^a)-\left\{\theta^a_{op}, \theta^c_{op}\right\}\| _{\ell^\star(\Lambda_a\times\Lambda_c)},
\end{split}
\end{equation}
where $\|\cdot\|_{\ell^\star(\Lambda_a\times\Lambda_c)}$ is the norm induced by (\ref{eq:bform}) and
\begin{equation}\label{eq:qNorm}
\| Q \| = \sup_{\{\mu^a,\mu^c\}\in\Lambda_a\times\Lambda_c}
\frac{\| Q\left(\left\{\mu^a,\mu^c\right\}\right) \|_{\ell^2(\Omega)}}
{\| \left\{\mu^a,\mu^c\right\} \| _{\ell^\star(\Lambda_a\times\Lambda_c)}}\,.
\end{equation}
The first term in (\ref{est0}) is the consistency error of the operator $P$.
Using (\ref{projectionrestrict})
\begin{equation}\label{est3}
\| \tilde{u}^a - P\left(r(\tilde{u}^a)\right) \|_{\ell^2(\Omega)} = \| \tilde{u}^a - u^c\| _{\ell^2(\Omega_c\backslash\Omega_o)}\,,
\end{equation}
i.e., the consistency error is confined to the purely continuum region.
The second term is proportional, up to a factor of $\|Q\|$,
to the approximation error\footnote{This error measures the difference between traces of the true atomistic solution $\tilde{u}^a$
and the approximate AtC solution (\ref{eq:atcapprox}).} in the solution of the
reduced space problem (\ref{unconstrained}). We proceed with an estimate of the approximation error, followed by a bound on the operator norm $\|Q\|$.
\begin{lemma}\label{trace}
Let $\tilde{u}^a$ solve~\eqref{testLabel0} and $\left\{\theta^a_{op},
\theta^c_{op}\right\}$ be the minimizer of~\eqref{unconstrained}.  Then
\begin{equation}\label{est13}
\| r(\tilde{u}^a) - \left(\left\{\theta^a_{op}, \theta^c_{op}\right\}\right)\| _{\ell^\star(\Lambda_a\times\Lambda_c)}
\leq
\| \tilde{u}^a -u^c \| _{\ell^2(\Omega_o)}\,.
\end{equation}
\end{lemma}
\begin{proof}
We bound the approximation error directly by noting $\left\{\theta^a_{op},
\theta^c_{op}\right\}$ solves the Euler-Lagrange
equation~\eqref{varForm} of the reduced space problem. As a
result,
\begin{equation}\label{est1a}
\begin{split}
\| &r(\tilde{u}^a) - \left(\left\{\theta^a_{op}, \theta^c_{op}\right\}\right)\| _{\ell^\star(\Lambda_a\times\Lambda_c)} \\
&\qquad = \sup\limits_{\left\{\mu^a, \mu^c\right\} \neq 0}\frac{\left|\langle r(\tilde{u}^a),
\left\{\mu^a, \mu^c\right\}\rangle
+\left(u^{a,0}-u^{c,0},  v^a(\mu^a) - v^c(\mu^c)\right)_{\ell^2(\Omega_o)}\right|}
{\| \left\{\mu^a, \mu^c\right\}\| _{\ell^\star(\Lambda_a\times\Lambda_c)}}.  \\
\end{split}
\end{equation}
Using Definition  (\ref{eq:bform}), \eqref{ua} and
\eqref{uc}, we can obtain
\begin{equation*}
\begin{split}
&\langle r(\tilde{u}^a),\left\{\mu^a, \mu^c\right\}\rangle  + \left(u^{a,0}-u^{c,0},  v^a(\mu^a) - v^c(\mu^c)\right)_{\ell^2(\Omega_o)}
\\[1ex]
&\qquad=
\left( v^a(r^a(\tilde{u}^a))-v^c(r^c(\tilde{u}^a)), v^a(\mu^a) - v^c(\mu^c)\right)_{\ell^2(\Omega_o)}
\\
&\qquad\qquad+
\left(u^{a,0}-u^{c,0},  v^a(\mu^a) - v^c(\mu^c)\right)_{\ell^2(\Omega_o)}
\\[1ex]
&\qquad= \left( v^a(r^a(\tilde{u}^a))+u^{a,0}-v^c(r^c(\tilde{u}^a))-u^{c,0}, v^a(\mu^a) - v^c(\mu^c)\right)_{\ell^2(\Omega_o)}
\\[1ex]
&\qquad= \left( \tilde{u}^a - u^c, v^a(\mu^a) - v^c(\mu^c)\right)_{\ell^2(\Omega_o)}
\le
\|\tilde{u}^a - u^c\|_{\ell^2(\Omega_o)}\cdot
\|\{\mu^a,\mu^c\}\|_{\ell^*(\Lambda_a\times\Lambda_c)}
 .
\end{split}
\end{equation*}
Using this identity in (\ref{est1a}) completes the proof.
\end{proof}
\smallskip
%\end{comment}

The following lemma estimates the norm of $Q$.
\begin{lemma}\label{Q}
%There exists a constant $c_1,$ independent of $K,\, L,$ and $N,$ such that
Under Assumption~\ref{assumpA} and Assumption~\ref{assumpB}, the norm of $Q$ is bounded by
\begin{equation}\label{eq:estimateQ}
\| Q\| \lesssim \gamma^{-1}\sqrt{\frac{N}{L-K}},
\end{equation}
where the implied constant is allowed to depend on $p$.
\end{lemma}
\begin{proof}
%%%
Definition (\ref{eq:qNorm}) implies that (\ref{eq:estimateQ}) will follow if we can show that
\begin{equation}\label{qEstimate22}
\| Q\left(\{\mu^a, \mu^c\} \right)\| _{\ell^2(\Omega)}
\lesssim~ \gamma^{-1}\sqrt{\frac{N}{L-K}}\,
\| \left\{\mu^a,\mu^c\right\} \| _{\ell^\star(\Lambda_a\times\Lambda_c)}
\quad\forall \, \{ \mu^a, \mu^c\}\in\Lambda_a \times \Lambda_c.
\end{equation}
On the other hand, the definition of $Q$ and (\ref{eq:bform}) imply that (\ref{qEstimate22}) is equivalent to
\begin{equation}\label{qEstimate22a}
\| v^a(\mu^a) \|^2_{\ell^2(\Omega_a)} +
\| v^c(\mu^c) \|^2_{\ell^2(\Omega_c/\Omega_o)}
\lesssim~ \gamma^{-2} \left(\frac{N}{L-K}\right)
\| v^a(\mu^a) - v^c(\mu^c)\| ^2_{\ell^2(\Omega_o)}
\end{equation}
for all $\{ \mu^a, \mu^c\}\in\Lambda_a \times \Lambda_c.$ To prove  (\ref{qEstimate22a}) we use the structure of $v^a(\mu^a)$ and $v^c(\mu^c)$.

Recall that  $v^{c}(\mu^c)\in\mathcal{U}_c$ solves the continuum submodel
$$
C_cv^c=0\quad \mbox{in $\Omega^\circ_o$}, \qquad v^c_K= \mu^c_K,\qquad v^c_{\bar{N}}=0\,.
$$
It is straightforward to verify that $v^{c}(\mu^c)$ is a linear function, i.e.,
\begin{equation}\label{contsol}
v^c_i=\alpha_c\frac{\bar{N}-i}{\bar{N}-K}\,,
\end{equation}
where $\alpha_c = \mu^c_K$.

%%%

On the other hand, $v^{a}(\mu^a)\in\mathcal{U}_a$ solves the atomistic submodel
$$
A_av^a=0\quad\mbox{in $\Omega^\circ_a$}
\qquad v^a_0=v^a_1=0,
\qquad v^a_{L-1}= \mu^a_{L-1},
\qquad  v^a_{L}= \mu^a_{L}\,.
$$
We decompose this field as  $v^a(\mu^a)=v^1 + v^2 + v^3 + v^4$
where
\begin{gather}\label{eq:modes}
v^1_i=\alpha_1 \frac iL\quad\mbox{and}\quad v^2_i=\alpha_2
\lambda^{L-i}\sqrt{L-K}
\end{gather}
are the dominant linear and exponential modes with $\alpha_1$ and $\alpha_2$ determined by the boundary conditions below.  Meanwhile, $v^3$ and $v^4$
are corrections to ensure that $v^a(\mu^a)=0$ on $\Gamma^{-}_a$, i.e.,
\begin{equation}
\left\{
\begin{array}{rll}
A_av^3  = &  0    & \mbox{in $\Omega^\circ_a$}  \\
      v^3  = & -v^2 & \mbox{on $\Gamma^{-}_a$} \\
      v^3  =&   0    & \mbox{on $\Gamma^{+}_a$}
\end{array}
\right.
\quad\mbox{and}\quad
\left\{
\begin{array}{rll}
A_av^4  = &  0    & \mbox{in $\Omega^\circ_a$}  \\
      v^4  = & -v^1 & \mbox{on $\Gamma^{-}_a$} \\
      v^4  = &   0    & \mbox{on $\Gamma^{+}_a$}
\end{array}
\right.
\,.
\end{equation}
To obtain  $v^2$ we need the roots of the characteristic polynomial of $A_a$
\[
p(\sigma) = -k_2\sigma^4 - k_1\sigma^3 + (2k_1 + 2k_2)\sigma^2 - k_1\sigma - k_2,
\]
which are given by  (see~\cite{dobsonluskin09})
\[
\lambda_1=\lambda_2=1;\quad \lambda_{3,4} = \frac{k_1 + 2k_2 \pm \sqrt{k_1^2 + 4k_1k_2}}{-2k_2}.
\]
We define $v^2$ by setting $\lambda=\lambda_4 = \frac{k_1 + 2k_2 -\sqrt{k_1^2 + 4k_1k_2}}{-2k_2}$ in (\ref{eq:modes}). Note that $0<\lambda<1$, as seen from the assumptions $k_1 > 0, k_2 < 0$, and $k_1 + 4k_2 > 0$.

The coefficients $\alpha_1$ and $\alpha_2$ are uniquely determined from the boundary condition $v^a(\mu^a)=\mu^a$ on $\Gamma^+_a$, which yields the following $2\times 2$ system:
\begin{equation}\label{eq:cov}
\begin{split}
\left(1-\frac{1}{L}\right)\alpha_1 + \lambda\sqrt{L-K} {\alpha_2} =~& \mu^a_{L-1} ,\\
{\alpha_1} + \sqrt{L-K}{\alpha_2}  =~&  \mu^a_{L}.
\end{split}
\end{equation}
Recall that {$\alpha_c = \mu^c_K$}. In the following, we define
\[
\alpha:=\sqrt{\alpha_1^2+\alpha_2^2+{\alpha_c^2}}.
\]
According to Remark \ref{rem:bounds}, the result of Lemma~\ref{stab} applies to $v^3$ and $v^4$, and so,
$$
\| v^3\| ^2_{\ell^2(\Omega_a)} \lesssim~ L  \| v^2 \|^2_{\ell^2(\Gamma^{-}_a)}
\quad\mbox{and}\quad
\| v^4\| ^2_{\ell^2(\Omega_a)} \lesssim~ L  \| v^1 \|^2_{\ell^2(\Gamma^{-}_a)} \,.
$$
Since $ v^1 \big |_{\Gamma^{-}_a}=\left(0,{\alpha_1}/L \right)$
and $ v^2 \big
|_{\Gamma^{-}_a}={\alpha_2}\sqrt{L-K}\left(\lambda^{L},\lambda^{L-1}
\right),$ we have the bounds
\begin{equation}\label{errorThrow}
\| v^3\| _{\ell^2(\Omega_0)}^2 \lesssim \lambda^{2(L-1)}L(L-K) \alpha_2^2	,
\quad\mbox{and}\quad
\| v^4\| _{\ell^2(\Omega_0)}^2 \lesssim  L\frac{\alpha_1^2}{L^2}.
\end{equation}
Using (\ref{errorThrow}) yields the following lower bound for the right hand side in (\ref{qEstimate22a}):
\begin{equation}\label{say}
\begin{split}
\| v^a(\mu^a) -  v^c(\mu^c) & \| _{\ell^2(\Omega_0)} \geq \| v^1 + v^2 - v^c\| _{\ell^2(\Omega_0)} - \| v^3\| _{\ell^2(\Omega_0)} - \| v^4\| _{\ell^2(\Omega_0)} \\[1ex]
\gtrsim ~& \| v^1 + v^2 - v^c\| _{\ell^2(\Omega_0)}
- \left(\lambda^{L-1}\sqrt{L(L-K)}
+ \frac{1}{\sqrt{L}}\right)\alpha.
\end{split}
\end{equation}

We proceed with estimating
\begin{equation}\label{eq0:ap}
\| v^1 + v^2 - v^c\| _{\ell^2(\Omega_0)}^2
= \| v^1 - v^c\|^2_{\ell^2(\Omega_0)} + 2\left( v^1 - v^c, v^2\right)_{\ell^2(\Omega_0)} + \| v^2\|^2_{\ell^2(\Omega_0)}.
\end{equation}
The term $\| v^1 - v^c\|^2_{\ell^2(\Omega_0)}$ is similar to the term in (\ref{eq:bform}) defining the trace norm $\|\cdot\|_{\ell^*(\Lambda_a\times\Lambda_c)}$, but it is simpler in that both $v^1$ and $v^c$ solve continuum
problems.  We will prove in Appendix~\ref{linAp} that
\begin{equation}\label{eq1:ap}
\| v^1 - v^c\|^2 _{\ell^2(\Omega_0)} \gtrsim (L-K) \gamma^2 \left(\alpha_c^2 + \alpha_1^2\right)
\end{equation}
for large $N.$
Furthermore, summing a finite geometric series shows
\begin{equation}\label{eq2:ap}
\| v^2 \|^2_{\ell^2(\Omega_0)} = \frac{L-K}{1-\lambda^2}\left(1-\lambda^{2(L-K+1)}\right)\alpha_2^2.
\end{equation}

Intuitively, one should suspect the cross term, $( v^1 -
v^c, v^2)_{\ell^2(\Omega_0)}$, in~\eqref{eq0:ap} to be estimable for large
overlap widths since the exponential term $v^2$ is not well
approximated by any linear function. %Since for large $N$ and
%$i\in\Omega_0$  we have $\frac{N-i}{N-K} \sim 1$,
We now calculate via explicit summation that
\begin{equation}\label{eq3:ap}
\begin{split}
\big( v^1 - v^c&, v^2\big)_{\ell^2(\Omega_0)}
=~  \sum_{i=K}^L \left({\alpha_1}\frac{i}{L}-\alpha_c \frac{\bar{N}-i}{ \bar{N} -K}\right){\alpha_2}\lambda^{L-i}\sqrt{L-K}
\\
=~&{\alpha_2}\sqrt{L-K}
\left(
{\alpha_1}\sum_{i=K}^L  \frac{i}{L}\lambda^{L-i} -
\alpha_c \sum_{i=K}^L \frac{ \bar{N} -i}{ \bar{N} -K}\lambda^{L-i}
\right)
\\
\geq~&
-|\alpha_1\alpha_2|\sqrt{L-K} \sum_{i=K}^L  \lambda^{L-i}
-|\alpha_c\alpha_2|\sqrt{L-K} \sum_{i=K}^L \lambda^{L-i}
\\
\gtrsim~&
-\sqrt{L-K}\left(\frac{1-\lambda^{L-K+1}}{1-\lambda}\right)\alpha^2 .
%%
%%
%%\big( v^1 - v^c&, v^2\big)_{\ell^2(\Omega_0)}
%%\\
%%=~& \sqrt{L-K}\alpha_2\alpha_3\left( \frac{i}{L}, \lambda^{L-i}\right)_{\ell^2(\Omega_0)} - \sqrt{L-K}\alpha_1\alpha_3\left( \frac{N-i}{N-K}, \lambda^{L-i} \right)_{\ell^2(\Omega_0)} \\
%%\geq~& -\frac{\sqrt{L-K}}{2}\left((\alpha_2^2 + \alpha_3^2)\left( \frac{i}{L}, \lambda^{L-i}\right)_{\ell^2(\Omega_0)} + (\alpha_1^2+\alpha_3^2)\left( 1, \lambda^{L-i}\right)_{\ell^2(\Omega_0)}\right) \\
%%\geq~& -\frac{\sqrt{L-K}}{2}\left(\alpha_1^2 + \alpha_2^2 + 2\alpha_3^2\right)\left( 1, \lambda^{L-i}\right)_{\ell^2(\Omega_0)} \\
%%\geq~& -\sqrt{L-K}\alpha^2\left(\frac{1-\lambda^{1+L-K}}{1-\lambda}\right).
\end{split}
\end{equation}
Using the inequalities~\eqref{eq1:ap},~\eqref{eq2:ap}, and~\eqref{eq3:ap} in~\eqref{eq0:ap} produces
\begin{equation}\label{eq4:ap}
\begin{split}
\| v^1 &+  v^2  - v^c\| _{\ell^2(\Omega_0)}^2
\gtrsim (L-K)\gamma^2(\alpha_c^2 + \alpha_1^2)
\\[0.75ex]
&
+(L-K)\left(\frac{1-\lambda^{2(L-K+1)}}{1-\lambda^2}\right) \alpha_2^2
-\sqrt{L-K}\left(\frac{1-\lambda^{L-K+1}}{1-\lambda}\right)\alpha^2
\\[0.75ex]
=~&
(L-K)\left[
\Big(\gamma^2-\frac{1}{\sqrt{L-K}}\cdot\frac{1-\lambda^{L-K+1}}{1-\lambda}
\Big)(\alpha_c^2 + \alpha_1^2)
\right.
\\[0.75ex]
&\qquad\qquad\qquad\quad
+\left.
\Big(
\frac{1-\lambda^{2(L-K+1)}}{1-\lambda^2}
-
\frac{1}{\sqrt{L-K}}\cdot\frac{1-\lambda^{L-K+1}}{1-\lambda}
\Big)
\alpha^2_2
\right].
\end{split}
\end{equation}
For a sufficiently large overlap region, there holds
$$
\sqrt{L-K} > \max
\left\{
\frac{1-\lambda^{L-K+1}}{\gamma^2(1-\lambda)},\frac{1+\lambda}{1+\lambda^{L-K+1}}
\right\},
$$
which guarantees the positivity of the terms multiplying $\alpha_c^2 + \alpha_1^2$ and $\alpha^2_2$ above.
Then since $\gamma^2 < 1$, we obtain from (\ref{eq4:ap}) that
\begin{equation}\label{eq5:ap}
\| v^1 + v^2 - v^c\| _{\ell^2(\Omega_0)}^2 \gtrsim (L-K)\gamma^2\alpha^2.
\end{equation}
Similarly, using (\ref{eq5:ap})  in~\eqref{say} yields
\begin{equation}\label{eq6:ap}
\| v^a - v^c\| _{\ell^2(\Omega_0)} \gtrsim~ \sqrt{L-K}\gamma\alpha.
\end{equation}
To complete the proof, we use the above results to estimate the
left-hand side
 in (\ref{qEstimate22a}):
\begin{align*}
&
\| v^a \|^2_{\ell^2(\Omega_a)} +  \| v^c\| _{\ell^2(\Omega_c\backslash\Omega_o)}^2
=
 \| v^1 + v^2 + v^3 + v^4\| ^2_{\ell^2(\Omega_a)} + \| v^c\| _{\ell^2(\Omega_c\backslash\Omega_o)}^2
\\[1ex]
&\quad\leq  4\left(
\| v^1\| _{\ell^2(\Omega_a)}^2 +
\| v^2\| _{\ell^2(\Omega_a)}^2 +
\| v^3\| _{\ell^2(\Omega_a)}^2 +
\| v^4\| ^2_{\ell^2(\Omega_a)} \right)+
\| v^c\| _{\ell^2(\Omega_c\backslash\Omega_o)}^2
\\
&\quad\leq 4\left(
L\alpha_1^2 +
(L-K)\frac{1-\lambda^{2(1+L-K)}}{1-\lambda^2} \alpha_2^2     +
L(L-K)\lambda^{2L-2} \alpha_2^2  +
\frac{\alpha_1^2}{L}\right) \\&\qquad+
( \bar{N} -L){\alpha_c^2} \\
&\quad\lesssim N(\alpha_c^2 + \alpha_1^2 + \alpha_2^2) \\
%\\
%&\qquad
&\quad\lesssim \frac{N \gamma^{-2}}{L-K}\| v^a - v^c\| _{\ell^2(\Omega_0)}^2.
\end{align*}
Estimates of the norms of $v^1$, $v^2$, $v^3$, $v^4$, and $v^c$ follow from
(\ref{stabbound}) in Lemma \ref{stab},  (\ref{eq2:ap}) and (\ref{errorThrow}), respectively. The final inequality, which establishes the assertion of the lemma, is a consequence of \eqref{eq6:ap}.
\end{proof}
\smallskip

%%%%%%%%%%%%LEMMA%%%%%%%%%%%%%%%%%%%%%%%%%%%%%%%
%\begin{lemma}\label{lem:qTerm}
%Let $r(\tilde{u}^a)$ and $\left\{\theta^a_{op}, \theta^c_{op}\right\}$ be as above.
%%Then for large enough $L$,
%Then there exists a constant $c_1,$ independent of $K,\,L,$ and
%$N,$ such that
%\begin{equation}\label{qEstimate}
%\| Q\left(r(\tilde{u}) - \left\{\theta^a_{op}, \theta^c_{op}\right\}\right)\| _{\ell^2} \leq c_1\sqrt{\frac{N}{L-K}}\| \tilde{u}^a - u^c\| _{\ell^2(\Omega_o)}.
%\end{equation}
%\end{lemma}
%\begin{proof}
%The proof follows directly from Lemma~\ref{trace} and
%Lemma~\ref{Q}.
%\end{proof}
%\smallskip
%%%%%%%%%%%%%%%%%%%%%%%%%%%%%%%%%%%%%%%%%%%%%%%%%%%%%%%%%%%

All results necessary for the completion of the AtC approximation error bound in
~\eqref{est0} are now in place.
\begin{prop}\label{res1}
Let $\tilde{u}^a$ solve~\eqref{testLabel0} and $\left\{\theta^a_{op},
\theta^c_{op}\right\}$ be the minimizer of~\eqref{unconstrained}.  The AtC solution satisfies the error bound
\begin{equation}\label{finale}
\| \tilde{u}^a - u^{atc}\| _{\ell^2(\Omega)}
\lesssim \left(1 + \gamma^{-1}\sqrt{\frac{N}{L-K}}\right)\| \tilde{u}^a - u^c\| _{\ell^2(\Omega_c)}.
\end{equation}
\end{prop}
\begin{proof}
Recall the split of the AtC solution error into a consistency error due to $P$ and the approximation error in the
reduced space problem (\ref{unconstrained}):
\begin{equation}
\begin{split}
\| \tilde{u}^a - &u^{atc}\| _{\ell^2(\Omega)} =
\| \tilde{u}^a - P\left(\left\{\theta^a_{op}, \theta^c_{op}\right\}\right)\| _{\ell^2(\Omega)}  \\[0.5ex]
\leq~& \| \tilde{u}^a - P(r(\tilde{u}^a))\| _{\ell^2(\Omega)} +
   \| Q\| \cdot
   \| r(\tilde{u}^a)-\left\{\theta^a_{op}, \theta^c_{op}\right\}\| _{\ell^\star(\Lambda_a\times\Lambda_c)},
\end{split}
\notag
\end{equation}
Using (\ref{est3}) for the consistency error, (\ref{est13}) for
the approximation error, and (\ref{eq:estimateQ}) for operator
norm yields the result of the proposition.
\end{proof}
\smallskip

Proposition~\ref{res1} reveals that the accuracy of the AtC
approximation is determined by two independent factors.
Replacing the atomistic model with a continuum model on
$\Omega_c$ introduces the \emph{continuum modeling error} $\|
\tilde{u}^a - u^c\| _{\ell^2(\Omega_c)}$, which is independent
of the choice of the coupling mechanism. An inherent assumption
in atomistic-to-continuum coupling is that the continuum model
closely approximates the atomistic model in the continuum
region, which can be expected when there are no
defects in the continuum region~\cite{acta.atc}. Thus, we expect $\| \tilde{u}^a - u^c\| _{\ell^2(\Omega_c)}$
to be small so long as this assumption holds.
On the other hand, the coupling mechanism via  the optimization framework
introduces the prefactor
$$
\sqrt{\frac{N}{L-K}}\approx \sqrt{\frac{|\Omega|}{|\Omega_o|}}\,,
$$
which depends on the size of the overlap region.
As can be expected, the AtC error is inversely proportional
to the size of $\Omega_o$.

We can precisely estimate the modeling error by applying the estimate
\eqref{CBounda} to the domain $\Omega_c.$ The operator $C$ is
now the $1D$ discrete Laplacian on $\Omega_c$ with homogeneous
Dirichlet boundary conditions at $K$ and $N-1.$ Thus, the
minimum eigenvalue for $C$ is $\lambda_1 =
4k_c\sin^2\left(\frac{\pi}{2(n+1)}\right)$ where now $n =
N-K-2$ is the dimension of $C$.  We then have since $\tilde{u}^a-u^c=0$
at atoms $K$ and $N-1$ that
\begin{equation}\label{CBound}
\begin{split}
\| \tilde{u}^a - u^c\| _{\ell^2(\Omega_c)}
\lesssim~& (N-K)^2\| \Delta_1^2\tilde{u}^a\| _{\ell^2(\Omega_c)}.
\end{split}
\end{equation}
The estimate (\ref{CBound}) confirms that the modeling error is small whenever
$\tilde{u}^a$ is smooth over the continuum region in the sense that
$\| \Delta_1^2\tilde{u}^a\| _{\ell^2\left(\Omega_c\right)}$ is small.

By using the modeling error bound~\eqref{CBound} in (\ref{finale}),  we obtain the following theorem for the AtC error estimate.
\begin{thm}\label{res2}
Under the conditions of Assumptions \ref{assumpA} and \ref{assumpB},  let $\tilde{u}^a$ solve~\eqref{testLabel0}, and let $\{\theta^a_{op},
\theta^c_{op}\}$ be the minimizer of~\eqref{unconstrained}.  Then
\begin{equation}\label{finale1}
\begin{split}
\| \tilde{u}^a - u^{atc}\| _{\ell^2(\Omega)}
\lesssim~& \left(1 + \gamma^{-1}\sqrt{\frac{N}{L-K}}\right)(N-K)^2\| \Delta_1^2\tilde{u}^a\| _{\ell^2\left(\Omega_c\right)}
,
%\\
%\leq~& c\frac{N^{\frac{5}{2}}}{\sqrt{L-K}}\| \Delta_1^2\tilde{u}^a\| _{\ell^2\left(\Omega_c\right)}.
\end{split}
\end{equation}
where $u^{atc}=P(\{\theta^a_{op}, \theta^c_{op}\})$ is the AtC solution and $P$ is defined by \eqref{projectionDef}.
\end{thm}

Note that the dependence of the AtC error on the size of the overlap domain is unchanged,
i.e., as the overlap width is increased, the
error decreases.  In the present situation, we did not
coarse-grain the continuum region, so the only increase in
complexity comes from increasing the size of the atomistic and
continuum regions.
%One could conceivably coarse-grain the
%continuum model without increasing the complexity by also
%allowing the overlap region to be coarse grained.
%However, the analysis provided here would not apply.

%
\subsection*{Thermodynamic limit}
By letting $N \to \infty$, the problem above is an example of a
thermodynamic limit.
A further estimate would typically be obtained
by assuming the fully atomistic solution decays sufficiently
rapidly as $N \to \infty$.  See \cite{acta.atc} for an analysis
in this setting for quasicontinuum methods.

We may conversely introduce an interatomic spacing parameter
$\epsilon$ with $\epsilon$ dependent norm
\[
\| u\| _{\epsilon} = \sqrt{\epsilon\sum u_i^2}.
\]
Setting $\epsilon = {N}^{-1}$ maintains
$\Omega=[0,1]$ and scaling the lattice by $j \mapsto \epsilon j$
scales
\[
\| \Delta_1^2u\| _{\ell^2\left(\Omega_c\right)} \mapsto \epsilon^4\| \Delta_1^2u\| _{\ell^2\left(\Omega_c\right), \epsilon}.
\]
The estimate in~\eqref{finale1} under this scaling is
\begin{equation}\label{scaleEst}
\| \tilde{u}^a - u^{atc}\| _{\ell^2(\Omega), \epsilon}
\lesssim~ \frac{\epsilon^{\frac{3}{2}}}{\sqrt{L-K}}\| \Delta_1^2\tilde{u}^a\| _{\ell^2\left(\Omega_c\right), \epsilon},
\end{equation}
which is the scaling limit as $\epsilon \to 0$. See~\cite{DobsonLuskinOrtner2010b}
for the derivation of a similar estimate in the case of the force based quasicontinuum operator in which $\epsilon$ is maintained as a parameter throughout.  Recalling Assumptions \ref{assumpA} and \ref{assumpB}, if the
 overlap region $\Omega_o$ has width $|\Omega_o|:=(L-K)\epsilon^{1/p}$ in the
 scaling limit, then we obtain the bound
\begin{equation}\label{scaleEst1}
\| \tilde{u}^a - u^{atc}\| _{\ell^2(\Omega), \epsilon}
\lesssim~  \frac{\epsilon^{\frac{3}{2} + \frac{1}{2p}}}{|\Omega_0|^{\frac{1}{2}}}\| \Delta_1^2\tilde{u}^a\| _{\ell^2\left(\Omega_c\right), \epsilon}.
\end{equation}
Hence, we may achieve any power of $\epsilon$ in the interval $\left(\frac{3}{2}, 2\right)$ for $p > 1$.
\section{Conclusion}\label{sec:concl}
This paper formulates and analyzes a new, optimization-based
strategy for atomistic-to-continuum coupling. Specifically, we
pose the problem of coupling a non-local, atomistic description
of a material with a local, continuous description as a
constrained optimization problem. The objective is to minimize
the $\ell^2$ difference between the continuum and atomistic
displacement fields over an overlap region, subject to
constraints expressing the atomistic and continuum force
balances in the respective subregions. The traces of the
atomistic and continuum solution components on the boundary of
the overlap region act as virtual boundary controls. Thus, our
approach can be viewed as an extension of the heterogeneous
decomposition method~\cite{gervasio_2001} to the AtC context.

%%%%%%%%%%%%%%%%%%%%%%%%%%%%%%%%%%%%%%%%%%%%%%%%%%%%%%%%%%%%%%%%%%%%%%
\section*{Acknowledgments}
The work of P. Bochev was supported by the Applied Mathematics
Program within the Department of Energy (DOE) Office of
Advanced Scientific Computing Research (ASCR). Part of this
research was carried under the auspices of the Collaboratory on
Mathematics for Mesoscopic Modeling of Materials (CM4), The
work of D. Olson was partially supported by Sandia's Computer
Science Research Institute Summer Internship Program.
%%%%%%%%%%%%%%%%%%%%%%%%%%%%%%%%%%%%%%%%%%%%%%%%%%%%%%%%%%%%%%%%%%%%%%
% Need lower case "appendix" to label with A, B, etc
\appendix

\section{Stability of atomistic and continuum problems}\label{A}
In this appendix, we prove the result stated in Lemma~\ref{stab}:
%there exists a constant $c$ independent of $K,\,L,$ and $N$ such that
\begin{equation}\label{stabboundAp}
\begin{split}
\| v^a(\theta^a)\| ^2_{\ell^2(\Omega_a)} \lesssim~& L   \|\theta^a\|^2_{\ell^2(\Gamma^+_a)} ,\\
%cL\left( (\theta^a_L)^2+(\theta^a_{L-1})^2 \right),\\
\| v^c(\theta^c)\| ^2_{\ell^2(\Omega_c)} \le~& (N-K)   \|\theta^c\|^2_{\ell^2(\Gamma^-_c)}.
%(\theta^c_K)^2.
\end{split}
\end{equation}
The second bound is a direct consequence of the maximum principle for the continuum operator $C = -k_c\Delta_1$.  Recalling that $v^c(\theta^c)$ is zero on $\Gamma_c^+$ and equal to $\theta^c_K$ on $\Gamma_c^- = \left\{K\right\}$, we have
\begin{equation}
\| v^c(\theta^c)\| ^2_{\ell^2(\Omega_c)} = \sum_{i = K}^{N-1} \left(v^c_i\right)^2 \leq \sum_{i = K}^{N-1} \left(\theta^c_K\right)^2 = (N-K)(\theta^c_K)^2.
\end{equation}

To prove the first bound in~\eqref{stabboundAp}, we note that the atomistic solution, $v^a(\theta^a)$, may be written as
\begin{equation}\label{atSolution}
\begin{split}
v^a(\theta^a)_n &= \beta_1\frac{n}{L} + \beta_2\frac{L-n}{L} + \beta_3\lambda^n + \beta_4\lambda^{L-n} \\
&=: \beta_1v^1(\theta^a) + \beta_2v^2(\theta^a) + \beta_3v^3(\theta^a) + \beta_4v^4(\theta^a),
\end{split}
\end{equation}
where $0 < \lambda < 1$ was defined in Lemma~\ref{Q} and the
coefficents are determined via the boundary conditions
$v^a(\theta^a) = 0$ on $\Gamma_a^-$ and $v^a(\theta^a) =
\theta^a$ on $\Gamma_a^+$.  Specifically,
\begin{equation}\label{coefficients}
T_L\begin{pmatrix}
\beta_1 \\
\beta_2 \\
\beta_3 \\
\beta_4
\end{pmatrix} :=
\begin{pmatrix}
0 &1 &1 &\lambda^L \\
\frac{1}{L} &\frac{L-1}{L} &\lambda &\lambda^{L-1} \\
\frac{L-1}{L} & \frac{1}{L} &\lambda^{L-1} & \lambda \\
1 &0 &\lambda^L & 1
\end{pmatrix}
\begin{pmatrix}
\beta_1 \\
\beta_2 \\
\beta_3 \\
\beta_4
\end{pmatrix} =
\begin{pmatrix}
0 \\
0 \\
\theta^a_{L-1} \\
\theta^a_L
\end{pmatrix},
\end{equation}
where
\begin{equation}
T_L \to \begin{pmatrix}
0 &1 &1 &0 \\
0 &1 &\lambda &0 \\
1 & 0 &0 & \lambda \\
1 &0 &0 & 1
\end{pmatrix} =: T \, \,  \mbox{as} \, \,  L \to \infty.
\end{equation}
For any $\delta > 0$, we can therefore choose $L$ such that $\| T_{L}^{-1}\|  < \| T^{-1}\|  + \delta$, and hence
\begin{equation}\label{coefEstimate}
\begin{split}
\left(\beta_1^2 + \beta_2^2 + \beta_3^2 + \beta_4^2\right) \leq \left(\| T^{-1}\|  + \delta\right)^2\left(\left(\theta^a_{L-1}\right)^2+ \left(\theta^a_L\right)^2 \right).
\end{split}
\end{equation}
Using successive Cauchy inequalities and explicit summation of finite geometric series yields
\begin{equation}\notag
\begin{split}
\| v^a  (\theta^a & )  \|^2_{\ell^2(\Omega_a)} \\
\leq~&
 4\left(\beta_1^2\| v^1(\theta^a)\| ^2_{\ell^2(\Omega_a)} \!+ \beta_2^2\| v^2(\theta^a)\| ^2_{\ell^2(\Omega_a)} \!+ \beta_3^2\| v^3(\theta^a)\| ^2_{\ell^2(\Omega_a)}
\!+ \beta_4^2\| v^4(\theta^a)\| ^2_{\ell^2(\Omega_a)}\right) \\
\lesssim~& 4\left(\beta_1^2L + \beta_2^2L + \beta_3^2\left(\frac{1 - \lambda^{L+1}}{1 - \lambda}\right) + \beta_4^2\left(\frac{{\lambda}^{-1}-\lambda^L}{\lambda^{-1}-1}\right)\right) \\
\lesssim~& L\left(\beta_1^2 + \beta_2^2 + \beta_3^2 + \beta_4^2\right)  \\
\leq~& \left(\| T^{-1}\|  + \delta\right)^2L\left(\left(\theta^a_{L-1}\right)^2+ \left(\theta^a_L\right)^2 \right),
\end{split}
\end{equation}
for large enough $L$ by~\eqref{coefEstimate}.

%%%%BEGINNING OF APPENDIX B FOR LINEAR ESTIMATE
\section{Estimate of $\| v^1 - v^c\|_{\ell^2(\Omega_0)}$}\label{linAp}
Finally, we establish the estimate \eqref{eq1:ap} under the conditions of Assumptions \ref{assumpA} and \ref{assumpB}.
Recall that $v^c$ and $v^1$ are defined in (\ref{contsol}) and (\ref{eq:modes}), respectively, and so,
\begin{equation}\label{eq:formz}
\|v^c - v^1\|^2_{\ell^2(\Omega_0)}
= \sum_{i=K}^{L}\left( \alpha_c  \frac{ \bar{N} -i}{ \bar{N} -K} - {\alpha_1}\frac{i}{L}\right)^2
= \tilde{A}\alpha^2_c -2\tilde{C}\alpha_c\alpha_1 + \tilde{B}\alpha^2_1
\end{equation}
where the coefficients of the quadratic form in (\ref{eq:formz}) are given by
$$
\tilde{A} = \sum_{i=K}^{L}\left( \frac{ \bar{N} -i}{ \bar{N} -K} \right)^2,\quad
\tilde{B} = \sum_{i=K}^{L}\left( \frac{i}{L}\right)^2,\quad\mbox{and}\quad
\tilde{C} = \sum_{i=K}^{L}\left( \frac{ \bar{N} -i}{ \bar{N} -K} \right)\cdot \left(\frac{i}{L}\right)\,,
$$
respectively. Summing the finite series for each coefficient
yields $\tilde{A} = \beta\cdot A$, $\tilde{B} = \beta\cdot B$,
and $\tilde{C}=\beta\cdot C$ where the common factor is
$\beta=(1+L-K)$ and
\begin{align*}
A=~&\frac{6 \bar{N}^2 + 2 L^2 + 2 K^2  - 6 K \bar{N} - 6 L \bar{N} + 2 K L +L-K}{6(K-\bar{N})^2},
\\
B=~&\frac{2L^2 + 2K^2 + 2K L  + L- K}{6L^2},
\\
C=~&\frac{2L^2 + 2K^2+ 2K L  -  3 K \bar{N} - 3 L \bar{N} + L - K}{6L(K-\bar{N})}\,,
\end{align*}
respectively.
Using that $K=(1-\gamma) L$ allows us to further write the coefficients as
$$
A=\frac{
	\bar{N}^2
	+L^2\big(1-\gamma+\frac13\gamma^2\big)
	+\frac16 \gamma L
	-L\bar{N}(2-\gamma)}{(\bar{N}-(1-\gamma) L)^2},
$$
$$
B=\frac{L(1-\gamma+\frac13\gamma^2\big)+\frac16 \gamma}{L},\quad\mbox{and}\quad
C=\frac{
	\bar{N}\big(1-\frac12 \gamma\big)
	-\frac16 \gamma
	-L(1-\gamma+\frac13 \gamma^3)}{\bar{N}-(1-\gamma) L}.
$$
Assumption \ref{assumpA} implies that $\lim_{N\rightarrow\infty} L/N =0$, and so
$$
A \to A_\infty = 1
, \quad
B \to B_\infty = {\textstyle 1-\gamma+\frac13\gamma^2}
, \quad\mbox{and} \quad
C \to C_\infty = {\textstyle 1-\frac12\gamma}\,.
$$
Let $0\le\lambda_1\leq\lambda_2$ be the eigenvalues of the
quadratic form ${A_\infty}\alpha^2_c -2{C_\infty}\alpha_c\alpha_1 +
{B_\infty}\alpha^2_1$. Using the expressions for the determinant and
the trace of the quadratic form, $\lambda_1 \lambda_2 = A_\infty B_\infty-C_\infty^2$
and $\lambda_1+\lambda_2 = A_\infty+B_\infty,$ we can estimate
$$
\lambda_1
= \frac{\lambda_1 \lambda_2}{\lambda_2}
\geq
\frac{\lambda_1 \lambda_2}{\lambda_1+\lambda_2}
=
\frac{\frac{1}{12} \gamma^2}{1 + (1-\gamma+\frac13\gamma^2)}
\geq \frac{1}{24} \gamma^2.
$$
This completes the proof.

%%%%%END OF LINEAR ESTIMATE

%%%%%%%%%%%%%%%%%%%%%%%%%%%%%%%%%%%%%%%%%%%%%%%%%%%%%%%%%%%%%%%%%%%%%%
\newpage
\bibliographystyle{plain}	% (uses file "plain.bst")
\bibliography{Manuscript27}		% expects file "myrefs.bib"
%%%%%%%%%%%%%%%%%%%%%%%%%%%%%%%%%%%%%%%%%%%%%%%%%%%%%%%%%%%%%%%%%%%%%%

\end{document}